\documentclass[12pt,a4paper]{amsart}
\usepackage[latin1]{inputenc}
\usepackage[english]{babel}
\usepackage{amsmath,amstext,amsthm,amsfonts,amssymb,amscd}
\usepackage[dvips]{graphicx}
\usepackage{hyperref}
\usepackage{color}
\usepackage{cleveref}

\theoremstyle{plain}
\newtheorem{theorem}{Theorem}

\newtheorem{definition}{Definition}
\newtheorem{remark}{Remark}

\newtheorem{lemma}{Lemma}
\newtheorem{corollary}{Corollary}
\newtheorem{theorema}{Theorem}

\begin{document}
	
	\title[Equilibrium States for Random Zooming Systems]{Equilibrium States for Random Zooming Systems}

	\author[R. A. Bilbao]{Rafael A. Bilbao}
	\address{Rafael A. Bilbao, Universidad Pedag\'ogica y Tecnol\'ogica de Colombia, Avenida
		Central del Norte 39-115, Sede Central Tunja, Boyac\'a, 150003, Colombia. 
	}
	\email{rafael.alvarez@uptc.edu.co}
	\author[M. Oliveira]{Marlon Oliveira}
	\address{Marlon Oliveira, Departamento de Matem\'atica e Inform\'atica, Universidade Estadual do Maranh\~ao,  65055-310  S\~ao Lu\'is, Brazil} \email{marlon.cs.oliveira@gmail.com}
	
	\author[E. Santana]{Eduardo Santana}
	\address{Eduardo Santana, Universidade Federal de Alagoas, 57200-000 Penedo, Brazil}
	\email{jemsmath@gmail.com}
	
	\dedicatory{Dedicated to Paulo Ribenboim for his outstanding contribution to Mathematics}
	
	\thanks{The third author was partially supported by CNPq,   under the project with reference 409198/2021-8.} 
	
	\thanks{The authors have no competing interests to declare that are relevant to the content of this article.}
	
	\thanks{The data avaibility of this paper includes the preprint on arXiv:2401.09437}

	\date{\today}

	\maketitle
	
	\begin{abstract} 
		In this work, based on \cite{Pi1} for deterministic systems, we extend the notion of zooming systems to the random context and based on the technique of \cite{AMO} we prove the existence of equilibrium states for which we call random zooming potentials, that include the hyperbolic ones, possibly with the presence of a critical set. With a mild condition, we obtain uniqueness. As an example of existence, we have the so-called random Viana maps with critical points. We also prove that the classes of random zooming potentials and random hyperbolic potentials are equivalent and also contain the null potential, giving measures of maximal entropy.
	\end{abstract}

	\bigskip


	
	\section{Introduction}
	The theory of equilibrium states on dynamical systems was firstly developed by Sinai, Ruelle and Bowen in the sixties and seventies. 
	It was based on applications of techniques of Statistical Mechanics to smooth dynamics.
	Given a continuous map $f: M \to M$ on a compact metric space $M$ and a continuous potential $\phi : M \to \mathbb{R}$, an  {\textit{equilibrium state}} is an 
	invariant measure that satisfies a variational principle, that is, a measure $\mu$ such that
	$$
	\displaystyle h_{\mu}(f) + \int \phi d\mu = \sup_{\eta \in \mathcal{M}_{f}(M)} \bigg{\{} h_{\eta}(f) + \int \phi d\eta \bigg{\}},
	$$
	where $\mathcal{M}_{f}(M)$ is the set of $f$-invariant probabilities on $M$ and $h_{\eta}(f)$ is the so-called metric entropy of $\eta$.

	In the context of uniform hyperbolicity, which includes uniformly expanding maps, equilibrium states do exist and are unique if the potential is H\"older continuous
	and the map is transitive. In addition, the theory for finite shifts was developed and used to achieve the results for smooth dynamics.
	
	Beyond uniform hyperbolicity, the theory is still far from complete. It was studied by several authors, including Bruin, Keller, Demers, Li, Rivera-Letelier, Iommi and Todd 
	\cite{BK,IT1,IT2,LRL} for interval maps; Denker and Urbanski  \cite{DU} for rational maps; Leplaideur, Oliveira and Rios 
	\cite{LOR} for partially hyperbolic horseshoes; Buzzi, Sarig and Yuri \cite{BS,Y}, for countable Markov shifts and for piecewise expanding maps in one and higher dimensions. 
	For local diffeomorphisms with some kind of non-uniform expansion, there are results due to Oliveira \cite{O}; Arbieto, Matheus and Oliveira \cite{AMO};
	Varandas and Viana \cite{VV}. All of whom proved the existence and uniqueness of equilibrium states for potentials with low oscillation. Also, for this type of maps,
	Ramos and Viana  \cite{RV}  proved it for potentials so-called \textit{hyperbolic}, which includes the previous ones. The hyperbolicity of the potential is
	characterized by the fact that the pressure emanates from the hyperbolic region. In all these studies the maps do not have the presence of critical sets and recently, Alves, Oliveira and Santana proved the existence of at most finitely many equilibrium states for hyperbolic potentials, possible with the presence of a critical set (see \cite{AOS}). More recently, Santana generalized this result by for general zooming contractions, that is, beyond the context of exponential contractions in \cite{S}. Moreover, in this work it is proved that the class of hyperbolic potentials is equivalent to the class of continuous zooming potential (which satisfies a key inequality between free energies). It includes the null potential, which implies existence and uniqueness of measure of maximal entropy. A similar work is developed in \cite{PV}, where they deal with expanding measures and potentials.
	
	In the context of random dynamics, some works has established existence and others uniqueness. We can cite \cite{AMO,BR,K,SSV}. All these works deal with dynamics with absence of critical sets. In this work, we generalize some of these results and extend them to the context of zooming contractions, obtaining random zooming systems. Based on the definition of zooming maps, introduced in \cite{Pi1} and developed equilibrium states for the deterministic context in \cite{S}, we introduce random zooming systems.

	We define what we mean by random zooming systems (which generalizes both deterministic zooming systems and random non-uniformly expanding maps), random zooming potentials and follow ideas of \cite{AMO} to obtain equilibrium states. With a mild condition, we obtain uniqueness. Moreover, based on ideas of \cite{S} we prove that the class of random zooming potentials is equivalent to the class of random hyperbolic potentials and contain the null potential, giving measures of maximal entropy.
	
	In the Example section, the systems considered are classical in the literature but the novelty is the existence of equilibrium states. We give the class of random Viana maps and construct random zooming potentials for maps with the presence of a fixed point in the random zooming set, which is the case of the random Viana maps. Finally, we prove that for random systems which have equilibrium states that are not measures of maximal entropy at the same time, we have the null potential as a random zooming potential.
	We also present an example of random zooming system which is not random non-uniformly expanding.
	\section{Preliminaries and Main Result}

	\subsection{Random Maps and Invariant Measures}
	
	Let $M$ be a connected compact metric space and $\mathcal{C}^{0}:=\mathcal{C}^{0}(M,M)$ the space of continuous maps on $M$. Let $(\Omega,T,\mathbb{P})$ be a measure preserving system, where $T : \Omega \to \Omega$ is $\mathbb{P}$-invariant (where $\mathbb{P}$ is a Borel measure) and $\Omega$ is a Polish space, i.e., $\Omega$ is a complete separable metric space. By a random transformation we mean a continuos map $f: \Omega \to \mathcal{C}^0$. Then, we define the skew-product generated by $f$:
	$$  
	\begin{array}{cccc}
		F \ : & \! \Omega \times M & \! \longrightarrow
		& \! \Omega \times M \\
		& \! (w,x) & \! \longmapsto
		& \! (T(w),f_{w}(x)) 
	\end{array}
	$$
	where the iterates of $F$ are given by 
	
	$$
	F^{n}(w, x) = (T^{n}(w), f^{n}_{w}(x))
	$$
	with, 
	\[
	f^{n}_{w}(x):= f_{T^{n-1}(w)} \circ \dots \circ f_{w}(x).
	\]

	
	In this work, we aim to study the dynamics of the orbits $(f_w^n(x))$ for each $x\in M$ and $w\in \Omega$. In other words, we analyze the dynamics of $F$ on its fibers $\{w\}\times M:=M_w$, conditioned by the dynamics of the system $T$.
	
	We denote by $\mathcal{P}_{\mathbb{P}}(\Omega \times M)$ the space of probability measures $\mu$ on $\Omega \times M$ such that the marginal of $\mu$ on $\Omega$ is $\mathbb{P}$. Let $\mathcal{M}(F, \Omega \times M) \subset \mathcal{P}_{\mathbb{P}}(\Omega \times M)$ the space of probability measure $\mu$ that are $F$-invariant. 
	
	Because $M$ is compact, invariant measures always exist and the property of $\mathbb{P}$ being the
	marginal on $\Omega$ of an invariant measure can be characterized by its disintegration:
	\[
	\int_{\Omega\times M} G(w,x) d\mu(w,x) =\int_{\Omega} \int_{M} G_{w}(x) d\mu_{w}(x)d\mathbb{P}(w),	
	\]
	where $\mu_{w}$ are called sample measures of $\mu$ (see \cite{L}) and $G:\Omega\times M \rightarrow \mathbb{R}$ a measurable function with $G_w:M\rightarrow \mathbb{R}$ measurable for each $w\in \Omega$ defined as $G_{w}(x):= G(w,x)$.
	
	\subsection{Random Zooming Systems}
	
	Based on the definition of zooming maps, introduced in \cite{Pi1} and developed equilibrium states for the deterministic context in \cite{S}, we introduce random zooming systems. The notion of zooming captures and weakens the geometric aspects of the {\it hyperbolic times}, that is, for exponential contractions (see \cite{Pi1}, section 8), allowing more flexibility in the applications and examples with nonexponential contractions.
	
	\begin{definition}
		(Zooming contractions). A \textit{zooming contraction} is a sequence of functions $\alpha_{n}: [0,+\infty) \to [0,+\infty)$ such that
		
		\begin{itemize}
			\item $\alpha_{n}(r) < r, \text{for all} \, \, n \in \mathbb{N}, \text{for all} \, \, r>0.$
			
			\item $\alpha_{n}(r)<\alpha_{n}(s), \,\, if \,\, 0<r<s, \text{for all} \, \, n \in \mathbb{N}$.
			
			\item $\alpha_{m} \circ \alpha_{n}(r) \leq \alpha_{m+n}(r), \text{for all} \, \, r>0, \text{for all} \, \, m,n \in \mathbb{N}$.
			
			\item $\displaystyle \sup_{r \in (0,1)} \sum_{n=1}^{\infty}\alpha_{n}(r) < \infty$.
		\end{itemize}
		
	\end{definition}
	
	As defined in \cite{PV}, we call the contraction $(\alpha_{n})_{n}$ \textit{exponential} if $\alpha_{n}(r) = e^{-\lambda n} r$ for some $\lambda > 0$ and \textit{Lipschitz} if $\alpha_{n}(r) = a_{n} r$ with $0 \leq a_{n} < 1, a_{m}a_{n} \leq a_{m+n}$ and $\sum_{n=1}^{\infty} a_{n} < \infty$. In particular, every exponential contraction is Lipschitz. We can also have the example with $a_{n} = (n+b)^{-a}, a > 1, b>0$.
	
	We define a random $(\alpha,\delta)$-{\it zooming time} for $(w,x) \in \Omega \times M$. Fix a zooming contraction $\alpha=\{\alpha_n\}$, $\delta > 0$ and $(w,x) \in \Omega \times M$, we say that $n \in \mathbb{N}$ is a $(\alpha, \delta)$-{\textit{random zooming time}} for $(w,x)$ if
	\begin{itemize}
		\item [(i)] There exists a neighbourhood $V_{n}(w,x) \subset \{w\} \times M$ of $(w,x)$ such that $f^{n}_{w}$ sends  $ V_{n}(w,x)$ homeomorphically onto the ball $ B_{\delta}(f^{n}_{w}(x)) \subset \{T^{n}(w)\} \times M$;
		\item [(ii)] $d(f^{i}_{w}(y),f^{i}_{w}(z)) \leq \alpha_{n-i}(d(f^{n}_{w}(y),f^{n}_{w}(z))), \forall \,\, y,z \in V_{n}(w,x), \forall \,\, 0 \leq i \leq n-1.$
	\end{itemize}
	
	We require that it holds for $\mathbb{P}$ almost every $w \in \Omega$. 
	
	The sets $V_{n}(w,x)$ are called \textit{random zooming pre-balls} and their images $f^{n}_{w}(V_{n}(w,x)) = B_{\delta}(f^{n}_{w}(x))$,  \textit{random zooming balls}.
	
	We observe that if $n$ is a $(\alpha,\delta)$-random zooming time for $(w,x)$, then $n$ is a $(\alpha,\delta')$-random zooming time for $(w,x)$, for every $0 < \delta' < \delta$.
	
	We say that $(w,x) \in \Omega \times M$ has {\textit{positive frequency}} of random zooming times if
	\begin{equation}
		\limsup_{n \to \infty}\frac{1}{n}\# \bigg\{0 \leq j \leq n-1 \,\, | \,\, j \,\, \text{is a random zooming time for} \,\, (w,x)\bigg\} > 0,
	\end{equation}\label{frequency}
	and define the {\textit{random zooming set}}
	$$
	\Lambda = \{(w,x) \in \Omega \times M \mid \text{the frequency of} \,\, (\alpha,\delta)- \text{random zooming times of} \,\, (w,x) \,\, \text{is positive}\}.
	$$
	We say that a Borel probability  measure $\mu$ on $\Omega \times M$ is {\textit{random zooming}} if $\mu(\Lambda)=1$.

	If the system $F:\Omega\times M \rightarrow \Omega \times M$, satisfies conditions (i) and (ii) and there exists a Borel probability measure $\mu$ random zooming, then $F$ it is called a {\it random zooming system}. Since the definition of zooming is a generalization of the definition of hyperbolic times, the existence of such systems is ensured.
	
	\begin{remark} We observe that for $\Omega=\{w\}$ the definition of random zooming system reduce to the notions of zooming system for deterministic cases as in \cite{Pi1}.
	\end{remark}

	\subsection{Entropy and Topological Pressure}
	We start by the definition of entropy and topological pressure for random transformations. The reader can consult more results and properties in Kifer~\cite{K} and Liu~\cite{L}.

	Let $\mu\in \mathcal{M}(F, \Omega\times M)$ be an $F$-invariant measure. Given a finite measurable partition $\xi$ of $M$ we set 
	$$h_{\mu}(f, \xi):=\lim_{n\rightarrow +\infty} \frac{1}{n} \int_\Omega H_{\mu_w}\left(\bigvee_{j=0}^{n-1}(f_{w}^{j})^{-1}(\xi)\right) d\mathbb{P}(w)$$
	where $H_{\nu}(\xi)=-\sum_{P\in \xi}\nu(P)\log\nu(P)$ for a finite partition $\xi$ and $\mu_w$ is the sample measure of $\mu.$
	The \emph{entropy} of $(f, \mu)$ is $$h_{\mu}(f):=\sup_{\xi}\{h_{\mu}(f,\xi)\}
	$$
	where the supremum is taken over all finite measurable partitions of $M$.
	
	Denote by $\mathbb{L}^{1}_{\mathbb{P}}(\Omega,C^0(M))$ the space of all measurable functions  $\phi:\Omega\times M\rightarrow \mathbb{R}$ such that $\phi_{w}:M\rightarrow \mathbb{R}$ defined by $\phi_{w}(x):=\phi(w,x)$ is continuous for all $w\in \Omega$ and $\lVert \phi \rVert_{1}=\int_{\Omega}\lVert \phi_{w}\rVert_{\infty}\,d\mathbb{P}(w)<+\infty$. 
	
	Fix $w\in \Omega$. Given $\varepsilon >0$ and an integer $n\geq 1$, we say that a subset $F_n\subseteq M$ is $(w,n,\varepsilon)-$separated if for every two distinct points $y,z\in F_n$ there exists some $j\in \{0,1,...,n-1\}$ such that $d(f^{j}_{w}(y),f^{j}_{w}(z))>\varepsilon$. 
	
	For $\phi \in  \mathbb{L}^{1}_{\mathbb{P}}(\Omega,C^0(M))$, $\varepsilon >0$ and $n\geq 1$ we consider 
	$$P_{f}(\phi)(w, n, \varepsilon)= \sup\left\{\sum_{y\in F_n}e^{ S_{n}\phi(w,y)} \ ;\ F_n\  \mbox{is a}\ (w,n,\varepsilon)-\mbox{separated set}\right\}$$
	where $S_{n}\phi(w,y):=\sum^{n-1}_{j=0}\phi_{\theta^{j}(w)}( f^{j}_{w}(y))$.
	
	The \emph{random topological pressure} of $\phi$ relative to $\theta$ is defined by	 
	$$P_{f}(\phi)=\lim_{\varepsilon\rightarrow 0}\limsup_{n\rightarrow \infty} \frac{1}{n}\int_{\Omega}\log P_{f}(\phi)(w, n, \varepsilon)\, d\mathbb{P}(w).$$
	Thus it is well defined the pressure map as follows
	$$
	\begin{array}{cccc}
		P_{f}\ : & \! \mathbb{L}^{1}_{\mathbb{P}}(\Omega,C^0(M)) & \! \longrightarrow
		& \!\mathbb{R}\cup\{\infty\} \\
		& \! \phi & \! \longmapsto	 & \! P_{f}(\phi)
	\end{array}
	$$ 
	In particular, the topological entropy of $F$ relative to $T$ is $h_{top}(f)=P_{f}(0)$.	
	
	The topological pressure and the entropy are related by the well known Variational Principle. The reader can see a proof of this result in \cite{L}.  
	
	\begin{theorem}	[Variational Principle]\label{thprinvaria}
		Assume that $(\Omega, \mathbb{P})$ is a Lebesgue space. Then for any $\phi\in \mathbb{L}^{1}_{\mathbb{P}}(\Omega,C^0(M))$ we have
		\begin{equation} 	\label{eqprinvaria}
			P_{f}(\phi)=\sup_{\mu\in \mathcal{M}_{\mathbb{P}}(F)}\biggl(h_{\mu}(f)+\int\phi\ d\mu \biggr).
		\end{equation}
		Moreover, when $\mathbb{P}$ is ergodic, we can consider the supremum over ergodic measures. 
	\end{theorem}
	
	Motivated by the Variational Principle, we say that an $F$-invariant measure $\mu\in \mathcal{M}(\Omega\times M)$ is an \emph{equilibrium state} for $(f, \phi)$ relative to $T$ if the supremum (\ref{eqprinvaria}) is attained by $\mu$, i.e.,
	\[
	P_{f}(\phi)=h_{\mu}(f)+\int\phi \ d\mu.
	\]

	\subsection{Topological Pressure and Variational Principle to noncompact set}
	
	Let $w\in \Omega$ and consider $\mathcal{U}$ a finite open cover of $M$. Denote by $\mathcal{S}_n(\mathcal{U})$ the set of all strings ${\bf{U}}=\{U_{i_0}, \cdots, U_{i_{n-1}}\,;\, U_{i_j}\in \mathcal{U}\}$ of length $n=n({\bf{U}})$ and put $\mathcal{S}=\mathcal{S}(\mathcal{U})=\bigcup_{n\geq 0}\mathcal{S}_n(\mathcal{U})$.
	Given a string ${\bf{U}}=\{U_{i_0}, \cdots, U_{i_{n-1}}\}\in\mathcal{S}(\mathcal{U})$,  we consider the cylinder 
	$$X_w=X_{w}({\bf{U}}):=\{x\in M\,\, ;\,\,  f^{j}_w(x)\in U_{i_j}\,\, \mbox{for}\,\, j=0, \cdots, n({\bf U})-1\}.$$
	Let $\mathcal{F}_{(N, w)}$ be the collection of all cylinders of depth at least $N$,  i.e., 
	$$\mathcal{F}_{(N, w)}=\mathcal{F}_{(N, w)}(\mathcal{U})=\{X_{w}({\bf{U}})\,\, ;\,\, {\bf{U}}\in \mathcal{S}_n(\mathcal{U})\,\, \mbox{for}\,\, n\geq N \}.$$
	For $\beta\in \mathbb{R}$ and $\phi \in  \mathbb{L}^{1}_{\mathbb{P}}(\Omega,C^0(M))$ let
	\begin{equation}\label{mensuravel}
		m_{\beta}(w, \phi, f, \mathcal{U}, N)=\displaystyle{\inf_{\mathcal{F}}}\left\{ \sum_{X_{w}\in\mathcal{F}_{(N, w)}} e^{-\beta n({\bf{U}}) + S_{n({\bf{U}})}\phi(X_w)}  \right\},
	\end{equation}
	where $S_{n({\bf{U}})}\phi(X_w)=\sup_{y\in X_w}\sum_{j=0}^{n({\bf U})-1}\phi_{T^{j}(w)}(f^{j}_{w}(y))$ and the infimum is taken over all finite families $\mathcal{F}$ of $\mathcal{F}_{(N, w)}$ in order that (\ref{mensuravel}) is measurable in $w$ (see e.g. Section~9 of \cite{SSV}). As $N$ goes to infinity we define 
	\[
	m_{\beta}(w, \phi, f, \mathcal{U})=\lim_{N\to \infty}m_{\beta}(w, \phi, f, \mathcal{U}, N).
	\]
	The existence of the limit above is guaranteed by the function $m_{\beta}(w, \phi, f, \mathcal{U}, N)$ to be increasing with $N$. Taking the infimum over $\beta$ we call
	\[
	P_{f}(w, \phi, \mathcal {U})=\inf\{\beta: m_{\beta}(w, \phi, f, \mathcal{U})=0 \}.
	\]
	Let $|\mathcal{U}|=\max\{\mbox{diam} U_i\,\,;\,\,U_i\subset\mathcal{U}\}$ be the diameter of the cover $\mathcal{U}$ and consider 
	$$P_{f} (w, \phi)=\lim_{|\mathcal{U}|\rightarrow 0} {P_{f} (w, \phi, \mathcal{U})}.$$
	In [Theorem~11.1, \cite{Pe2}] it was showed that this quantity is well defined and does not depend on the cover $\mathcal{U}$. Moreover, since all quantities defined above are measurable functions of $w\in \Omega$ (see e.g.~Section~9 of \cite{SSV}), we can define the \emph{random topological pressure} of $(f, \phi)$ as
	\[
	P_{f} (\phi)= \int P_{f} (w, \phi)\ d\mathbb{P}(w). 
	\]	
	
	In the following we present another way to define the random topological pressure. We fix $w\in \Omega$ and $\varepsilon>0$. For $n\in \mathbb{N}$, $x\in M$, let $B_{w}(x,n,\varepsilon)$ be the dynamic ball
	\[
	B_{w}(x,n,\varepsilon):=\{y\in M: d(f_{w}^{j}(x),f_{w}^{j}(y) )< \varepsilon, \ \mbox{ for} \ \ 0\leq j \leq n \}.
	\]
	We denote by $G_{(N,w)}$ the collection of dynamic balls:
	\[
	G_{(N,w)}:=\{ B_{w}(x,n,\varepsilon): x\in M \ \mbox{and} \ n\ge N \}.
	\]
	Let $U_w$ be a finite or countable family of $G_{(N,w)}$ which covers $M$. For every $\beta\in \mathbb{R}$ and $\phi \in  \mathbb{L}^{1}_{\mathbb{P}}(\Omega,C^0(M))$ let
	\[
	m_{\beta}(w, \phi, f, \varepsilon, N)=\inf_{U_{w}\subset G_{(N,w)}}\left\{ \sum_{B_{w}(x,n,\varepsilon)\in U_w} e^{-\beta n + S_{n}\phi(B_{w}(x,n,\varepsilon))}  \right\},
	\]
	where $S_{n}\phi(B_{w}(x,n,\varepsilon))=\sup_{y\in B_{w}(x,n,\varepsilon)}\sum_{j=0}^{n-1}\phi_{T^{j}(w)}(f^{j}_{w}(y))$. When $N$ goes to infinity we consider
	\[
	m_{\beta}(w, \phi, f, \varepsilon)=\lim_{N\to \infty}m_{\beta}(w, \phi, f, \varepsilon, N).
	\]
	and taking the infimum over $\beta$ we define
	\[
	P_{f}(w, \phi, \varepsilon)=\inf\{\beta: m_{\beta}(w, \phi, f, \varepsilon)=0 \}.
	\]
	Since $P_{f}(w, \phi, \varepsilon)$ is decreasing on $\varepsilon$ we can take the limit	$$P_{f} (w, \phi)=\lim_{\varepsilon\rightarrow 0} {P_{f} (w, \phi, \varepsilon)}.$$
	Now if we consider a finite open cover $\mathcal{U}$ of $M$ with Lebesgue number $\varepsilon(\mathcal{U})$ we have 
	$$B_w(x, n({\bf U}), \frac{1}{2}\varepsilon)\subset X_w({\bf U})\subset B_w(x, n({\bf U}), 2|\mathcal {U}|)$$
	which implies that $$P_{f} (w, \phi)=\lim_{\varepsilon\rightarrow 0} {P_{f} (w, \phi, \varepsilon)}=\lim_{|\mathcal{U}|\rightarrow 0} {P_{f} (w, \phi, \mathcal{U})}.$$
	Therefore the definition of random topological pressure via coverings and via dynamic balls coincides.
	
	The topological pressure and the entropy are related by the well known Variational Principle. The reader can see a proof of this result in \cite{SSV}.
	
	\begin{theorem}{(Variational Principle to noncompact sets).}\label{VP}
		Let $\phi\in \mathbb{L}^{1}_{\mathbb{P}}(\Omega,C^0(M))$ and let $\mu\in \mathcal{M}(\Omega\times M)$ be a probability measure  and $\Lambda \subset \Omega \times M$ such that $F(\Lambda)\subseteq \Lambda$. Then
		$$
		P_{f}(\phi,\Lambda) \geq \sup_{\mu(\Lambda)=1} \left\{ h_{\mu}(f) + \int_{\Lambda} \phi d\mu  \right\}.
		$$
	\end{theorem}

	\subsection{Random Zooming Potentials}
	
	Analogouly to what is given in \cite{PV}, we introduce the following definitions in the random context.
	
	Let $\mathcal{Z}$ be the set of random zooming measures on $\Omega \times M$. We say that a function $\phi  \in \mathbb{L}^{1}_{\mathbb{P}}(\Omega,C^0(M))$ is a \emph{random zooming potential} if the following inequality holds:
	\begin{equation}
		\label{equaexpandpoten1}
		\displaystyle \sup_{\nu \in \mathcal{Z}^{c}}\bigg{\{} h_{\nu}(f) + \int \phi d\nu \bigg{\}} < \sup_{\mu \in \mathcal{Z}}\bigg{\{} h_{\mu}(f) + \int \phi d\mu \bigg{\}}. 
	\end{equation}
	
	If $P(\phi)$ denotes the topological pressure of $\phi$, by the Variational Principle, we have:
	\[
	P(\phi) = \sup_{\mu \in \mathcal{M}(\Omega \times M)}\bigg{\{} h_{\mu}(f) + \int \phi d\mu \bigg{\}} = \sup_{\mu \in \mathcal{Z}}\bigg{\{} h_{\mu}(f) + \int \phi d\mu \bigg{\}}. 
	\]
	
	It means that, if the there exists some equilibrium state for a random zooming potential, the equilibrium states must be random zooming measures.

	\subsection{Random Hyperbolic Potentials}
	
	We say that a function $\phi  \in \mathbb{L}^{1}_{\mathbb{P}}(\Omega,C^0(M))$ is a \emph{random hyperbolic potential} if the topological pressure $P_{f}(\phi)$ is concentrated at the random zooming set $\Lambda$, that is, if the following inequality holds:
	\begin{equation}
		\label{equationradhyperpoten}
		P_{f}(\phi,\Lambda^{c}) < P_{f}(\phi,\Lambda) = P_{f}(\phi).
	\end{equation}
	As a consequence of the Theorem \ref{VP} (Variational Principle), we have the following inequality
	\[
	\displaystyle \sup_{\nu \in \mathcal{Z}^{c}}\bigg{\{} h_{\nu}(f) + \int \phi d\nu \bigg{\}} \leq P_{f}(\phi,\Lambda^{c}) < P_{f}(\phi,\Lambda) = P_{f}(\phi) = \sup_{\mu \in \mathcal{Z}}\bigg{\{} h_{\mu}(f) + \int \phi d\mu \bigg{\}}. 
	\]
	It means that every random hyperbolic potential is a random zooming potential.
	
	\subsection{Main Result}
	
	Our main result concerns the existence and uniqueness of equilibrium states for random zooming systems.
	
	\begin{theorema}
		\label{A}
		Let $F:\Omega \times M \to \Omega \times M$ be a random zooming map and $\phi : \Omega \times M \to \mathbb{R}$ a random zooming potential. Then,
		
		\begin{itemize}
			\item There exists some equilibrium state and every equilibrium state is a random zooming measure.
			
			\item Suppose that the marginal $\mathbb{P}$ is ergodic and the map $F : \Omega \times M \to \Omega \times M$ is such that the set $\{f^{n}_{w}(x) \,\, | \,\, n \in \mathbb{N}\}$ is dense on $M$, for every $(w,x) \in \Lambda \cap M_{w}$,  $\mathbb{P}$-a.e. $w \in \Omega$. In this case, all ergodic equilibrium states have same support $S$.
			
			\item If there exists a fixed point $y_{0} = (w_{0}, x_{0}) \in \Lambda \cap S$, we have uniqueness of equilibrium states.
		\end{itemize}
	\end{theorema}
	
	As a corollary of our main result concerns of the equivalence between the classes of random zooming potentials and random hyperbolic potentials.
	
	\begin{corollary}\label{B}
		We have that every random zooming potential is also a random hyperbolic potential. These two classes of potentials are then equivalent.
	\end{corollary}
	
	\section{Proof of Theorem \ref{A}}
	The following results can be found in  \cite{AMO}. The first of them shows the expansivity of the dynamics on the fibers. In the Lemma \ref{lemaparti} the existence of a generating partition is given over the set of measures $\mu$ that are in  $\mathcal{Z}$, which is the set of random zooming measures on $\Omega \times  M$. Consequently, the entropy over these measurements is equal; to the entropy considering the generating partition. Finally, the upper semicontinuity is obtained over a measure $\mu_0$.

	\begin{lemma}
		Let $\epsilon > 0$ and define
		\[
		A_{\epsilon}(w,x) = \{y \mid d(f^{n}_{w}(x),f^{n}_{w}(y)) \leq \epsilon, \forall n \geq 0\}.
		\]
		
		For every $\epsilon \leq \delta$ we have
		\[
		A_{\epsilon}(w,x) = \{x\},
		\]
		$\mathbb{P}$-almost every $w \in \Omega$.
	\end{lemma}
	\begin{proof}
		By equation \ref{frequency} we have that every point $(w,x) \in \Lambda$ possesses infinitely many $(\alpha,\delta)$-zooming times $n_{i} = n_{i}(w,x) \in \mathbb{N}$. For each $w \in \Omega$, set $\Lambda_{w} : = \{x \mid (w,x) \in \Lambda\}$. Then, $\mathbb{P}$-almost every $w \in \Omega$ we have $\mu_{w}(\Lambda_{w}) = 1$ and an infinity of $(\alpha,\delta)$-zooming times $\mu_{w}$-almost every $x \in \Lambda_{w}$. Also, for every $\epsilon \leq \delta$ we have that
		\[
		d(x,z) \leq \alpha_{n_{i}}(d(f^{n_{i}}_{w}(x),f^{n_{i}}_{w}(z))) \leq \alpha_{n_{i}}(\epsilon),
		\]
		because the contractions $\alpha_{n}$ are increasing functions. Also, since the following holds $\displaystyle \sup_{r \in (0,1)} \sum_{n=1}^{\infty}\alpha_{n}(r) < \infty$, we have that for $\epsilon \in (0,1)$ we obtain $\alpha_{n_{i}}(\epsilon) \to 0$. It implies that $d(x,z) = 0$ and we are done.
	\end{proof}
	
	Let $\mathcal{P}$ be a partition of $M$ in measurable sets with diameter less than $\delta$. From the above
	lemma, we state the following lemmas, which are useful for proving the main theorem and whose proofs are found in \cite{AMO}:
	
	\begin{lemma}\cite[Lemma 6.7]{AMO}
		\label{lemaparti}
		Let $\mathcal{P}$ be a partition of $M$ in measurable sets with diameter less that $\delta$. Then, $\mathcal{P}$ is a generating partition for every $\mu \in \mathcal{Z}$.
	\end{lemma}
	
	\begin{corollary}\cite[Lemma 6.8]{AMO}
		\label{corollary1}
		For every $\mu \in \mathcal{Z}$ we have $h_{\mu}(f) = h_{\mu}(F,\mathcal{P})$.
	\end{corollary}
	
	\begin{lemma}\cite[Lemma 6.9]{AMO}
		The map $\mu \to h_{\mu}(F,\mathcal{P})$ is upper semi-continuous at $\mu_{0}$ measure such that $(\mu_{0})_{w}(\partial P) = 0$ for $\mathbb{P}$-almost everywhere $w \in \Omega$, $P \in \mathcal{P}$.
	\end{lemma}
	
	Now, we are able to prove the Theorem \ref{A}.
	\begin{proof}
		Since the potential $\phi : \Omega \times M \to \mathbb{R}$ is random zooming, we can obtain a sequence of measures $\mu_{k} \in \mathcal{Z}$ such that
		\[
		h_{\mu_{k}}(f) + \int \phi d\mu_{k} \to P(\phi).
		\]
		
		Fix a partition $\mathcal{P}$ with diameter less than $\delta$ and for $w$-a.e., $\mu_{w}(\partial P) = 0$, for any $P \in \mathcal{P}$. By the above Corollary \ref{corollary1}, we have that $h_{\mu_{k}}(f)=h_{\mu_{k}}(F,\mathcal{P})$. Then, we obtain
		\[
		P(\phi) = \limsup_{k \to \infty} \bigg{\{} h_{\mu_{k}}(f) + \int \phi d\mu_{k}\bigg{\}}.
		\]
		
		We know that the map $\eta \to h_{\eta}(F,\mathcal{P})$ is upper semi-continuous over $\mathcal{Z}$ (also over $\overline{\mathcal{Z}}$). By compacity of the space of probabilities, we can suppose that there exists a measure $\mu \in \overline{\mathcal{Z}}$ such that 
		\[
		P(\phi) = \limsup_{k \to \infty} \bigg{\{} h_{\mu_{k}}(F,\mathcal{P}) + \int \phi d\mu_{k}\bigg{\}} \leq h_{\mu}(F,\mathcal{P}) + \int \phi d\mu \leq P(\phi).
		\]
		
		It means that $\mu$ is an equilibrium state for $\phi$.
	\end{proof}
	
	
	
	
	\begin{lemma}
		Suppose that the marginal $\mathbb{P}$ is ergodic and the map $F : \Omega \times M \to \Omega \times M$ is such that the set $\{f^{n}_{w}(x) \,\, | \,\, n \in \mathbb{N}\}$ is dense on $M$, for every $(w,x) \in \Lambda \cap M_{w}$,  $\mathbb{P}$-a.e. $w \in \Omega$. Then, all ergodic equilibrium states have same support.
	\end{lemma}
	\begin{proof}
		Let $\mu \neq \nu$ be two ergodic equilibrium states. Since $\mu(\Lambda) = 1 = \nu(\Lambda)$, because they are random zooming measures, we have that $\mu_{w}(\Lambda \cap M_{w})= 1$ and $ \nu_{w}(\Lambda \cap M_{w}) = 1$, $\mathbb{P}$-a.e. $w \in \Omega$. There exists $w_{0} \in \Omega$ such that $\mu_{w_{0}}(\Lambda \cap M_{w_{0}})= 1 = \nu_{w_{0}}(\Lambda \cap M_{w_{0}})$. We can also take $w_{0}$ such that the set $\{f^{n}_{w_{0}}(x) \,\, | \,\, n \in \mathbb{N}\}$ is dense on $M$. 
		
		Fix $w_{0} \in \Omega$. Given $\epsilon > 0$, Let $B_{0} \subset M$ be a ball with radius $\epsilon$. Take $(w_{0}, x_{0}) \in \Lambda \cap M_{w_{0}}$. There exists $n_{0} \in \mathbb{N}$ such that
		\[
		f^{n_{0}}_{w_{0}}(x_{0}) \in B_{0} \subset M_{T^{n_{0}}(w_{0})} \implies x_{0} \in f^{n_{0}}_{T^{n_{0}}(w_{0}))^{-1}}(B_{0}) \subset M_{w_{0}}.	
		\]
		It means that
		\[
		\Lambda \cap M_{w_{0}} \subset \bigcup_{i=1}^{\infty} f^{i}_{T^{i}(w_{0}))^{-1}}(B_{0}).
		\]
		Since $\mu_{w_{0}}(\Lambda \cap M_{w_{0}})= 1 = \nu_{w_{0}}(\Lambda \cap M_{w_{0}})$, there exist $r,s \in \mathbb{N}$ such that
		\[
		\mu_{w_{0}}(f^{r}_{T^{r}(w_{0}))^{-1}}(B_{0}))> 0, \nu_{w_{0}}(f^{s}_{T^{s}(w_{0}))^{-1}}(B_{0}))> 0,
		\]
		which implies, by invariance, that
		\[
		\mu_{T^{r}(w_{0})}(B_{0})> 0, \nu_{T^{s}(w_{0})}(B_{0})> 0 \implies \mu_{T^{t}(w_{0})}(B_{0})> 0, \nu_{T^{t}(w_{0})}(B_{0})> 0, t \geq \text{lcm}\{r,s\}.
		\]
		We claim that  $\mathbb{P}(\{w \in \Omega \,\, | \,\, \mu_{w}(B_{0}) > 0, \nu_{w}(B_{0}) > 0\}) = 1$.  In fact, this set is positively invariant. Since $\mathbb{P}$ is ergodic, we have $\mathbb{P}(\{w \in \Omega \,\, | \,\, \mu_{w}(B_{0}) > 0, \nu_{w}(B_{0}) > 0\}) = 0,1$. If $\mathbb{P}(\{w \in \Omega \,\, | \,\, \mu_{w}(B_{0}) > 0, \nu_{w}(B_{0}) > 0\}) =0$, we would have 
		\[
		w_{0} \in \bigcup_{i=1}^{\infty}T^{-i}(\{w \in \Omega \,\, | \,\, \mu_{w}(B_{0}) > 0, \nu_{w}(B_{0}) > 0\}) \,\, \mathbb{P}-\text{a.e.} \,\, w_{0} \in \Omega \,\, (*),
		\]
		and the union in $(*)$ with zero measure $\mathbb{P}$, by invariance. This is not possible and we obtain $\mathbb{P}(\{w \in \Omega \,\, | \,\, \mu_{w}(B_{0}) > 0, \nu_{w}(B_{0}) > 0\}) = 1$. So, $\mu(A \times B_{0}) > 0, \,\, \nu(A \times B_{0}) > 0$ for every set $A \subset \Omega$ such that $\mathbb{P}(A) > 0$ and we obtain $\text{supp}(\mu) = \text{supp}(\nu)$. 
	\end{proof}
	
	\begin{lemma}
		If there exists a fixed point $y_{0} = (w_{0}, x_{0}) \in \Lambda \cap S$, we have uniqueness of equilibrium states.
	\end{lemma}	
    \begin{proof}
     By \cite{OV}[Lemma 4.3.3] for every collection of ergodic probability measures $\{\eta_{\lambda}\}_{\lambda}$ there exists a pairwise disjoint collection of invariant subsets $P_{\lambda} \subset \Omega \times M$ such that $\eta_{\lambda}(P_{\lambda})= 1$. Let $\{\eta_{\lambda}\}_{\lambda}$ be the collection of ergodic equilibrium states of $\phi$. 
     
     Consider the following pairwise disjoint collection of invariant subsets $Q_{\lambda} : = P_{\lambda} \cap \Lambda \cap S$. We have $\eta_{\lambda}(Q_{\lambda}) = 1$ because $\eta_{\lambda}$ is a random zooming measure.
     
     Let $\Lambda\cap S = \bigcup \Lambda_{\lambda}$ be a decomposition into pairwise disjoint invariant subsets such that $\eta_{\lambda}(\Lambda_{\lambda}) = 1$. We have the fixed point $y_{0} \in \Lambda_{\lambda_{0}}$. If $\Lambda\cap S = \bigcup \Lambda_{\lambda}'$ is another decomposition, we have $y_{0} \in \Lambda_{\lambda_{0}'}'$.
     
     We have $\eta_{\lambda_{0}}(\Lambda_{\lambda_{0}} \cap \Lambda_{\lambda_{0}}') = 1$ and $\eta_{\lambda_{0}'}(\Lambda_{\lambda_{0}'} \cap \Lambda_{\lambda_{0}'}') = 1$. By considering the following decomposition $\Lambda\cap S = \bigcup (\Lambda_{\lambda} \cap \Lambda_{\lambda}')$, we have $y_{0} \in \Lambda_{\lambda''} \cap \Lambda_{\lambda''}'$. Once it holds that $y_{0} \in \Lambda_{\lambda_{0}} \cap \Lambda_{\lambda_{0}'}'$, we have $\lambda'' = \lambda_{0} = \lambda_{0}'$. So, for every decomposition $\Lambda\cap S = \bigcup \Lambda_{\lambda}$ we obtain $y_{0} \in  \Lambda_{\lambda_{0}}$. By taking the following decomposition,
     \[
     \Lambda\cap S = Q_{\lambda_{0}} \cup \bigcup_{\lambda \neq \lambda_{0}} \Lambda_{\lambda},
     \]
     we obtain $y_{0} \in Q_{\lambda_{0}}$. Hence, either we have uniqueness or we can decompose $\Lambda \cap S$ and  $\eta_{\lambda_{0}} = \delta_{y_{0}}$, the Dirac measure supported at $y_{0}$ and it is an equilibrium state of $\phi$. But all the ergodic equilibrium states have the same support $S$. In the case of $\delta_{y_{0}}$ being an equilibrium state, we obtain $S = \{y_{0}\}$ and all equilibrium state coincide with $\delta_{y_{0}}$. In any case, we have uniqueness. The Lemma is proved.
    \end{proof}

	\section{Proof of Corollary \ref{B}}
	
	In order to prove Corollary \ref{B}, we observe that both the sets $\mathcal{ZP}$ and $\mathcal{HP}$ of random zooming and random hyperbolic potentials, respectively, are open. Since we have that $\mathcal{HP} \subset \mathcal{ZP}$, we are done if we show that both open sets have same closure.
	
	We begin by observing that the closure of the random hyperbolic potentials is the set of potentials $\phi  \in \mathbb{L}^{1}_{\mathbb{P}}(\Omega,C^0(M))$ such that the following inequality holds:
	\begin{equation}
		\label{equationradhyperpoten}
		P_{f}(\phi,\Lambda^{c}) \leq P_{f}(\phi,\Lambda) = P_{f}(\phi).
	\end{equation} 
	
	In this case, for the potential $\phi : \Omega \times M \to \mathbb{R}$, we can also obtain a sequence of measures $\mu_{k} \in \mathcal{Z}$ such that
	\[
	h_{\mu_{k}}(f) + \int \phi d\mu_{k} \to P(\phi).
	\]
	It also provides some equilibrium state $\mu_{0}$ which is a random zooming measure. So, we have the equality:
	\[
	P(\phi) = h_{\mu_{0}}(f) + \int \phi d\mu_{0}.
	\]
	It implies that the potential $\phi$ is in the closure of the set $\mathcal{ZP}$ and $\overline{\mathcal{ZP}} \subset \overline{\mathcal{HP}}$. Since we have that $\overline{\mathcal{HP}} \subset \overline{\mathcal{ZP}}$, we obtain $\overline{\mathcal{HP}} = \overline{\mathcal{ZP}}$ and the proof is done.
	
	\section{Examples}
	
	In this section, we give examples of random zooming systems, as the random non-uniformly expanding maps so-called random Viana maps and construct zooming potentials. Random non-uniformly expanding maps are examples of random zooming systems when the contractions are exponential. In this case, random zooming times are called random hyperbolic times.
	
	We begin by constructing random zooming potentials for random maps which has zooming fixed points. The Dirac probability supported on this point will be the equilibrium state. After that, we show that the null potentials is random zooming, if we can find a potential with an equlibrium state which is not a measure of maximal entropy at the same time. Finally, we presente the class of random Viana maps and give an example of random zooming system which does not have exponential zooming contractions. 
	
	\subsection{Random Zooming Potentials}\label{fixed} Let $F$ be an random zooming map with a critical set such that the random zooming set $\Lambda$ contains a fixed point $x_{0}$. We will construct a random zooming potential, that is, a potential $\varphi$ such that
	
	\begin{equation*}
		\label{equaexpandpoten}
		\sup_{\nu \in \mathcal{Z}^{c}}\bigg{\{} h_{\nu}(f) + \int \varphi d\nu \bigg{\}} < \sup_{\mu \in \mathcal{Z}}\bigg{\{} h_{\mu}(f) + \int \varphi d\mu \bigg{\}},
	\end{equation*}
	
	where $\mathcal{Z}$ is the set of the zooming measures.
	
	Since $x_{0} \in \Lambda$ is a fixed point, the Dirac probability $\delta_{x_{0}}$ at $x_{0}$  is an invariant zooming measure.
	
	Let $\varphi \geq 0$ be a continuous potential on the fibers such that $\max \varphi = \varphi(x_{0}) > 0$ and it is integrable with respect to the marginal probability (for example, we can define it as constant on the fibers). We have that
	\[
	\int \varphi d\nu < \varphi(x_{0}) = \int \varphi d\delta_{x_{0}}, \forall \nu \in \mathcal{Z}^{c}.
	\]
	
	We assume that
	\[
	\sup_{\nu \in \mathcal{Z}^{c}}\bigg{\{} \int \varphi d\nu \bigg{\}} < \varphi(x_{0}).
	\]
	
	So, we can find $k > 0$ such that
	\[
	\int k\varphi d\delta_{x_{0}} - \sup_{\nu \in \mathcal{Z}^{c}}\bigg{\{} \int k\varphi d\nu \bigg{\}} > 2h(F) \implies \int k\varphi d\delta_{x_{0}} -  \int k\varphi d\nu > 2h(F), \forall \nu \in \mathcal{Z}^{c}  
	\]
	where $h(F) > 0$ is the topological entropy. It implies that
	\[
	h_{\nu}(f) + \int k\varphi d\nu \leq h(F) + \int k\varphi d\nu < 2h(F) + \int k\varphi d\nu < \int k\varphi d\delta_{x_{0}},
	\]
	that is,
	\[
	\sup_{\nu \in \mathcal{Z}^{c}}\bigg{\{} h_{\nu}(f) + \int k\varphi d\nu \bigg{\}} < \int k\varphi d\delta_{x_{0}} \leq \sup_{\mu \in \mathcal{Z}}\bigg{\{} h_{\mu}(f) + \int k\varphi d\mu \bigg{\}}.
	\]
	Then, the potential $\psi = k\varphi$ is random zoooming.

	Also, in general, if there exist a potential $\varphi$ and a measure $\mu \in \mathcal{Z}$, such that
	\[
	\sup_{\nu \in \mathcal{Z}^{c}}\bigg{\{} \int \varphi d\nu \bigg{\}} < \int \varphi d\mu,
	\]
	we can find $k > 0$ such that
	\[
	\int k\varphi d\mu - \sup_{\nu \in \mathcal{Z}^{c}}\bigg{\{} \int k\varphi d\nu \bigg{\}} > 2h(F) \implies \int k\varphi d\mu -  \int k\varphi d\nu > 2h(F), \forall \nu \in \mathcal{Z}^{c}  
	\]
	where $h(F) > 0$ is the topological entropy. It implies that
	\[
	h_{\nu}(f) + \int k\varphi d\nu \leq h(F) + \int k\varphi d\nu < 2h(F) + \int k\varphi d\nu < \int k\varphi d\mu,
	\]
	that is,
	\[
	\sup_{\nu \in \mathcal{Z}^{c}}\bigg{\{} h_{\nu}(f) + \int k\varphi d\nu \bigg{\}} < \int k\varphi d\mu \leq \sup_{\mu \in \mathcal{Z}}\bigg{\{} h_{\mu}(f) + \int k\varphi d\mu \bigg{\}}.
	\]
	Then, the potential $\psi = k\varphi$ is random zooming.

	\subsection{The Null Potential} We will show that for random maps and potentials for which there exist equilibrium states we have that the null potential is random zooming and, as a consequence, potentials with low variation are also random zooming. It is a consequence of the fact that the set of random zooming potentials is open.
	
	In order to show it, we use the existence of equilibrium states which are zooming measures. 
	
	We recall that $\mathcal{Z}$ denotes the set of zooming invariant measures. Assume, by contradiction, that
	\[
	\sup_{\mu \in \mathcal{Z}}\{h_{\mu}(f)\} \leq \sup_{\nu \in \mathcal{Z}^{c}}\{h_{\nu}(f)\}.
	\]
	Let $\phi_{0}$ be a random zooming potential which has equilibrium states. There exists an ergodic equilibrium state $\mu \in \mathcal{Z}$. We have that $\mu$ is also an equilibium state for $\phi: = \phi_{0} - P(\phi_{0})$ which implies that $P(\phi) = 0$ and $h_{\mu}(f) + \int \phi d\mu = 0$ and $-h_{\mu}(f) = \int \phi d\mu \leq 0$. Assume that there exists $\nu_{0} \in \mathcal{Z}^{c}$ such that $h_{\mu}(f) < h_{\nu_{0}}(f)$. Since every equilibrium state is a zooming probability, we have that $\nu_{0}$ is not an equilibrium state for $\phi$ and
	\[
	h_{\nu_{0}}(f) + \int \phi d\nu_{0} < h_{\mu}(f) + \int \phi d\mu = 0,
	\]
	which implies that $\int \phi d\nu_{0} < \int \phi d\mu = -h_{\mu}(f)$ because $h_{\mu}(f) < h_{\nu_{0}}(f)$.  For every $t \in (0,1)$ and putting $\nu = (1-t)\mu + t\nu_{0}$ we still have that $h_{\mu}(f) < h_{\nu}(f)$ (we observe that $h_{\nu}(f) = (1-t)h_{\mu}(f) + th_{\nu_{0}}(f)$, because the entropy is an affine function, see \cite[Proposition 9.6.1]{OV} for the deterministic case and by using the Abramov-Rokhlin Formula we obtain it for the relative entropy). Let $\epsilon > 0$ and $t \in (0,1), n \in \mathbb{N}$ are chosen such that such that $0 < h_{\nu}(f) - h_{\mu}(f) -\epsilon/n$ and  
	\[ 
	\int \phi d\mu > \int \phi d\nu > \int \phi d\mu - \epsilon 
	\]
	so, 
	\[
	- \epsilon(n-1)/n < h_{\nu}(f) - h_{\mu}(f) -\epsilon < h_{\nu}(f) + \int \phi d\nu < 0,
	\]
	using that $\int \phi d\mu = -h_{\mu}(f)$ . It is possible because once $\mu$ is an equilibrium state we have $h_{\nu_{0}}(f) - h_{\mu}(f) < \int \phi d\mu - \int \phi d\nu_{0}$ and there exist $t \in (0,1), n \in \mathbb{N}$ such that
	
	\begin{align*}
		\epsilon/n < h_{\nu}(f) - h_{\mu}(f) &= t(h_{\nu_{0}}(f) - h_{\mu}(f)) \\
		& < t\bigg{(}\int \phi d\mu - \int \phi d\nu_{0}\bigg{)} = \int \phi d\mu - \int \phi d\nu < \epsilon\\
	\end{align*}
	We obtain a contradiction because $\nu \in \mathcal{Z}^{c}$ and $\epsilon > 0$ is arbitrary and the potential $\phi$ is zooming. It shows that we have
	\[
	\sup_{\nu \in \mathcal{Z}^{c}}\{h_{\nu}(f)\} \leq h_{\mu}(f).
	\]
	If $\mu_{0}$ is an equilibrium state which is not a measure of maximal entropy, then we have
	\[
	\sup_{\nu \in \mathcal{Z}^{c}}\{h_{\nu}(f)\} \leq h_{\mu_{0}}(f) < \sup_{\mu \in \mathcal{Z}}\{h_{\mu}(f)\} = h(f).
	\]
	Hence, the null potential is random zooming.

	\subsection{Random Viana maps}
	
	
	
	Although this class of systems is classic in the literature, the novelty here is the existence of equilibrium states and measure of maximal entropy, since the null potential is random zooming. 
	
	Deterministic Viana maps are presented in \cite{A,V}. Alves and Ara\'ujo in \cite{AlvesAraujo}, considering $f_0$ as Viana map and the functions $f\in C^2$ around of Viana map must have the same critical point set $\mathcal{C}$ and impose that
	\[
	Df_{t}(x)=Df_0(x) \quad \text{for every} \quad x\in M\setminus \mathcal{C} \ \ \text{and} \ \ t\in Y,
	\]
	where $Y$ metric space with probability measure $\nu$. They also consider that, in the case of maps with critical sets, they assume {\it a slow approximation of random orbits to the critical set}: given any small $\gamma > 0$ there is $\delta >0$ such that
	$$
	\limsup_{n\to +\infty}\dfrac{1}{n}\sum_{j=0}^{n-1}-\log dist_{\delta} (f_{\underline{t}}^{j}(x), \mathcal{C}) \le \gamma
	$$
	for  $ \nu^{\mathbb{N}}\times m$ almost every $(\underline{t},x) \in \Omega\times M$ for $\Omega = Y^{\mathbb{N}}$. Where $dist_{\delta}(x,\mathcal{C})$  is $\delta-$truncated distance from $x\in M$ to $\mathcal{C}$, $\underline{t}=(t_1,t_2,\dots)\in \Omega$ and the random orbit $(f_{\underline{t}}^{n}(x))_{n\geq 1}$ defined
	$$
	f_{\underline{t}}^{n}= f_{t_n}\circ \cdots \circ f_{t_1} \quad \text{for} \quad n\ge 1
	. $$
	
	Now, under these conditions, they show that an RDS non-uniformly expanding with critical points is obtained, through the Proposition 2.6 (see \cite{AlvesAraujo}, pag 35), while at the Proposition 2.3 (see \cite{AlvesAraujo}, pag 32) the existence of hyperbolic times is obtained. \\
	On the other hand, considering $T:\Omega \rightarrow \Omega$  the left shift  map and $\nu^{\mathbb{N}}$ ergodic product measure on $(T, \Omega)$, then we have 
	
	\begin{enumerate}
		\item  [i)] A measure $\mu \in \mathcal{Z}$ which is random zooming: since the deterministic Viana map has an expanding fixed point and the shift map also has fixed points, as can obtain a fixed point for the random Viana map which is a zooming point. Also, the Dirac probability supported on this fixed zooming point is a random zooming measure; 
		
		\item [ii)] A potential $\varphi: \Omega\times M \rightarrow \mathbb{R}$  for the random Viana map which is a zooming potential (see \ref{equaexpandpoten1}). In the subsection \ref{fixed} we constructed such a potential which is random zooming, for any  random map which has a zooming fixed point.
		
		\item[iii)] The null potential is random zooming and we have the existence of measures of maximal entropy.
	\end{enumerate}
	
	
	\subsection{Random Zooming but not Expanding} We can obtain an example of a random zooming system which does not have exponential contractions by considering the product of the deterministic example given in \cite{Pi1} and the shift map of finite type. We obtain an skew product with contractions on the fibers which are not exponential. We refer \cite{Pi1}, Example 9.14. We can see that for this example we have the contractions as
	\[
	\alpha_{n}(r) = \bigg{(}\frac{1}{1 + n\sqrt{r}}\bigg{)}^{2}r.
	\]
	For the random system, we consider it between fibers.
	
	
	

\end{document}